\documentclass[11pt]{amsart}
\usepackage{amssymb,amsmath,amscd,amsthm,latexsym}
\usepackage{amssymb,graphicx}
\newtheorem{theorem}{Theorem}

\newtheorem{corollary}[theorem]{Corollary}

\newtheorem{lemma}[theorem]{Lemma}

\newtheorem{remark}[theorem]{Remark}

\title{Some Non-Amenable Groups}
\author{Aditi Kar and Graham A. Niblo}
\address{School of Mathematics, University of Southampton, Southampton, SO17 1BJ, UK}
\email{A.Kar@soton.ac.uk, G.A.Niblo@soton.ac.uk}

\thanks{Research partially supported by EPSRC grant  EP/F031947/1.}

\begin{document}
\begin{abstract} We generalise a result of R.~Thomas to establish the non-vanishing of the first $\ell^2$ Betti number for a class of finitely generated groups. \end{abstract}

\maketitle

\noindent In this note we give the following generalisation of a result of Richard Thomas \cite{Thomas}.

\begin{theorem} \label{main} Let $G$ be a finitely generated group given by the presentation \[\langle x_1, \ldots, x_d\ :\ u_1^{m_1}, \ldots,u_r^{m_r}\rangle\] such that each relator  $u_i$ has order $m_i$ in $G$. 
\begin{enumerate}
\item If $G$ is finite then $1 -d+\sum_{i=1}^r \frac{1}{m_i} > 0$ and $|G|\geq \frac{1}{1-d+\sum_{i=1}^r \frac{1}{m_i}}$.
\item If the first $\ell^2$ Betti number $\beta^2_1(G)$ of $G$ is zero, then $$1-d+\sum_{i=1}^r \frac{1}{m_i} \geq 0.$$ 
\end{enumerate}
\end{theorem}

In particular, the case when all the exponents $m_i$ in the presentation are equal to $1$ yields the well known observation that when the first $\ell^2$ Betti number is zero the deficiency of the presentation $d-r$ must be at most $1$. The vanishing of the first $\ell^2$ Betti number of a group $G$ holds for example if $G$ is finite, if it satisfies Kazhdan's property (T) or if it admits an infinite normal amenable subgroup (in particular if it is infinite amenable). We refer to \cite{Fernos} for other interesting examples. We obtain as a corollary:

\begin{corollary} \label{corollary} Let $G$ be a finitely generated group given by the presentation \[\langle x_1, \ldots, x_d\ :\ u_1^{m_1}, \ldots,u_r^{m_r}\rangle\] such that each relator  $u_i$ has order $m_i$ in $G$. If $d>1+\sum_{i=1}^r \frac{1}{m_i}$, then $G$ is infinite, does not satisfy Kazhdan's property (T) and has no amenable infinite normal subgroups.
\end{corollary}
 
Thomas established the inequality in (1) above by providing a simple but elegant computation of the dimension of  the  $\mathbb{F}_2$--vector space of 1-cycles of the cellular chain complex of the Cayley graph of $G$ (Thomas refers to this space as the cycle space of $\Gamma$). If $\Gamma$ has $d$ edges and $v$ vertices then the dimension of this vector space is $d-v+1$. An alternative approach, yielding information about the classical first Betti number of $G$ and its finite index subgroups is explored by Allcock in \cite{Allcock}.

We generalise this idea  to give the additional inequality in (2) above by using elementary observations about the $\ell^2$ Betti numbers $\beta^2_i$ of the orbihedral presentation 2-complex of $G$. For an introduction to $\ell^2$ Betti numbers, we refer the reader to \cite{eckmann}. The first $\ell^2$ Betti number vanishes for all finite groups. Cheeger and Gromov have shown that if a group $G$ is amenable then $\beta^2_1(G)=0$ \cite[Theorem 0.2]{cheeger}. More generally, $\beta^2_1(G)$ is zero for any group $G$ which contains an infinite normal amenable subgroup. 

We realised while writing this note that Theorem \ref{main} can be derived from the results of Peterson and Thom,  \cite{thom},  however our method is independent of theirs and provides a short and rather elementary proof of the result. 

\begin{remark} Equation (3) below yields the inequality $\beta^2_1(G)\geq \frac{1}{|G|} + d-1- \sum_i \frac{1}{m_i}$ from \cite{thom}. Here, $|G|$ denotes the size of $G$ and $\frac{1}{|G|}$ is understood to be zero when $G$ is infinite. \end{remark}

\noindent \textbf{Finitely generated but not finitely presented groups: } L\"{u}ck has defined $\ell^2$ Betti numbers for any countable discrete group. The notion agrees with the cellular $\ell^2$ Betti numbers for finitely presented groups and the basic properties including a generalised Euler-Poincar\'e formula for $G$-CW complexes may be found in Chapter 6 of \cite{luck}. Working in this context and arguing as in the proof of Theorem 1, we obtain the following generalisation. 

\begin{theorem}Suppose a group $G$ is given by the presentation  

\[
G=\langle x_1, \ldots, x_d\ :\ u_i^{m_i},\,  i \in I\rangle
\]

where $I$ is a countable set and each relator  $u_i$ has order $m_i$ in $G$. If $\sum_{i \in I} \frac{1}{m_i}$ converges then $\beta^2_1(G) \geq \frac{1}{|G|}+d-1-\sum_{i \in I} \frac{1}{m_i}$. In particular if $\beta^2_1(G)=0$ then $\sum_{i \in I} \frac{1}{m_i} -d+1 \geq 0$. \end{theorem}

Before we embark on the proof of Theorem \ref{main}, we need a short lemma which says that the orbihedral Euler characteristic of a $G$-CW complex $Y$ may be computed from its $\ell^2$ Betti numbers. The lemma is well known and may be found in \cite{luck}.  

\begin{lemma}\cite[Theorem 6.80]{luck} If $G$ acts on a connected CW complex $\tilde Y$ with finite quotient $Y$ such that stabilisers of cells are finite, then the $\ell^2$-Euler characteristic of $Y$ is equal to the orbihedral Euler characteristic of $Y$. More precisely, if for each $i$, $\Sigma_i$ is a choice of representatives for the orbits of $i$-cells in $\tilde Y$ and the stabiliser of a cell $\sigma$ in $G$ is written $G_\sigma$, then 
\begin{equation}\label{eq2}
\sum_i (-1)^i \beta^2_i(Y) = \sum_i (-1)^i \sum_{\sigma \in \Sigma_i} \frac{1}{|G_\sigma|}
\end{equation}
 \end{lemma}
 
\noindent We now proceed with the proof of Theorem \ref{main}.

\begin{proof}[Proof of Theorem \ref{main} ] Let $G$ be a group given by the presentation $\langle x_1,\ldots, x_d\ :\ u_1^{m_1},\ldots,u_r^{m_r}\rangle$ where each relator $u_i$ has order $m_i$ in $G$. The orbihedral presentation 2-complex of $G$, which we will denote by $\mathcal{P}$, has one vertex and $d$ edges forming a bouquet of $d$ circles. Identifying each of the circles with one of the generators $x_i$ we identify the fundamental group of this bouquet  with the free group on $\{x_1, \ldots, x_d\}$. Attached to this are $r$ discs, $\mathcal{D}_1,\ldots, \mathcal{D}_r$. For each $i=1,\ldots, r$, the disc $\mathcal D_i$ is endowed with a cone point of cone angle $\frac{2 \pi}{m_i}$ and its boundary is attached by a degree 1 map along the loop in the bouquet of circles corresponding to the element $u_i$. 

Attaching the corresponding stabilisers to cells we obtain, in the language of Haefliger \cite{haefliger}, a  developable complex of groups, meaning that the orbihedral universal cover $X$ of $\mathcal P$ exists. In fact, $X$ has a simple description in terms of the Cayley graph $\mathcal C$ of $G$. The  $1$-skeleton of the orbihedral universal cover is the Cayley graph of $G$ with respect to the generating set $\{x_1, \ldots, x_d\}$, while the $2$-skeleton is obtained from the 2-skeleton of the topological universal cover of  the presentation $2$-complex by collapsing stacks of relator discs having common boundaries. Specifically, the relator $u_i^{m_i}$ corresponds to a loop $\gamma_i$ in $\mathcal P$ bounding a disc and there is a unique lift $\tilde{\gamma_i}$ of $\gamma_i$ based at the identity vertex in $\mathcal C$. In the topological universal cover of the presentation 2-complex there are additional copies of this disc (glued along the same loop) based at the elements $u_i, \ldots u_i^{m_i-1}$ and the action of the subgroup $\langle u_i\rangle$ permutes these discs so that each  has trivial stabiliser. In contrast, these copies are identified in the orbihedral cover to give a single disc and it is preserved by the element $u_i$. The hypothesis that $u_i$ has order $m_i$ controls the order of the cell stabiliser. 

We now apply  the identity in (\ref{eq2}) to our complex $X$. The action of $G$ on the vertices and the edges of $X$ is both free and transitive. On the other hand, by hypothesis, the stabiliser of a lift of a $2$-cell $\mathcal D_i$ has order $m_i$. Hence, $\beta^2_0(\mathcal P) - \beta^2_1(\mathcal P) + \beta^2_2(\mathcal P) = 1 - d + \sum_i \frac{1}{m_i}$. We also know that $\beta^2_0(\mathcal P)= \frac{1}{|G|}$ where $\frac{1}{|G|}$ is understood to be zero when $G$ is infinite. Therefore, 

\begin{equation}\label{eq3}
\frac{1}{|G|} - \beta^2_1(\mathcal P) + \beta^2_2(\mathcal P) = 1 - d + \sum_i \frac{1}{m_i}.
\end{equation}

Finally we remark that the first $\ell^2$ Betti number of the group $G$ may be computed as the first $\ell^2$ Betti number of the orbihedral presentation complex used above. By definition, $\beta^2_1(G)$ is the von Neumann dimension of the first $\ell^2$ homology group of $Y$ with coefficients in the von-Neumann algebra of $G$, where $Y$ is the universal cover of the (topological) presentation 2 complex for $G$. Since both $X$ and $Y$ are simply connected we deduce from Theorem 6.54(3) of \cite{luck} that $\beta^2_1(G)= \beta^2_1(\mathcal P)$. Therefore, equation (\ref{eq3}) becomes 
\begin{equation}\label{eq4}
\frac{1}{|G|} - \beta^2_1(G) + \beta^2_2(\mathcal P) = 1 - d + \sum_i \frac{1}{m_i}.
\end{equation}

Now assume that $\beta^2_1(G)=0$. 
Since $\beta^2_2(\mathcal P) \geq 0$ , we get the identity we are looking for, namely \[1-d+\sum_{i=1}^r {\frac{1}{m_i}}\geq \frac{1}{|G|}.\] 
In particular, if $G$ is finite, then the $\ell^2$ cohomology of $G$ is just the group cohomology with real coefficients, and this vanishes so we obtain Thomas's result that $1 -d+\sum_{i=1}^r \frac{1}{m_i} > 0$ and $|G|\geq \frac{1}{1-d+\sum_{i=1}^r \frac{1}{m_i}}$. On the other hand, if $G$ is infinite and its first $\ell^2$ Betti number is zero, in particular if $G$ is an infinite amenable group, then we obtain the inequality $1-d+\sum_{i=1}^r {\frac{1}{m_i}}\geq 0$, as required. \end{proof}

\end{document}